\documentclass[a4paper,12pt]{amsart}
\usepackage{latexsym}
\usepackage{a4wide}
\usepackage{amscd}
\usepackage{graphics}
\usepackage{graphicx}
\usepackage[english]{varioref}
\usepackage{amsmath}
\usepackage{amssymb}
\usepackage{mathrsfs}
\usepackage{amsthm}
\usepackage{amssymb,tikz}
\usepackage{verbatim}
\usepackage{stmaryrd}
\usepackage[xindy]{glossaries}

\usepackage{tikz}
\tikzstyle{mybox} = [draw=black, very thick, rectangle, rounded corners, inner ysep=5pt, inner xsep=5pt]

\usepackage{bbm}
\usepackage{soul}
\usepackage{soul,color}
\usepackage{accents}
\usepackage{enumerate}
\usepackage[utf8x]{inputenc}
\usepackage[T1]{fontenc}
\usepackage[english]{babel}
\usepackage{amsfonts, amscd, amsmath, amssymb, amsthm, mathrsfs, mathtools}
\usepackage{braket}
\input xy
\xyoption{all}
\usepackage[utf8x]{inputenc}
\usepackage[T1]{fontenc}
\usepackage[english]{babel}
\usepackage{amsfonts, amscd, amsmath, amssymb, amsthm, mathrsfs, mathtools}
\usepackage{graphicx}
\usepackage{subfig}
\usepackage{braket}
\usepackage[left=2.5cm,right=2.5cm,top=2.5cm,bottom=2.5cm]{geometry}
\usepackage{fancyhdr}
\theoremstyle{definition}

\newtheorem{remark}{Remark}[section]

\newtheorem{esempio}{Example}[section]

\newtheorem{definizione}{Definition}[section]

\theoremstyle{plain}
\newtheorem{teorema}{Theorem}[section]
\newtheorem{proposizione}{Proposition}[section]
\newtheorem{lemma}{Lemma}[section]

\newcommand{\numberset}{\mathbb}

\newcommand{\R}{\numberset{R}}
\newcommand{\N}{\numberset{N}}
\newcommand{\Z}{\numberset{Z}}

\usepackage{bookmark,hyperref}

\DeclarePairedDelimiter{\abs}{\lvert}{\rvert}
\DeclarePairedDelimiter{\norma}{\lVert}{\rVert}

\makeatletter
\let\oldabs\abs
\def\abs{\@ifstar{\oldabs}{\oldabs*}}
\let\oldnorma\norma
\def\norma{\@ifstar{\oldnorma}{\oldnorma*}}

\title{A Conley-type Lyapunov function for the strong chain recurrent set}

\author[Bernardi]{Olga Bernardi$^1$}
\address{$^2$Dipartimento di Matematica ``Tullio Levi-Civita'', Università di Padova, Via Trieste 63, 35121 Padova, Italy}
\email{obern@math.unipd.it}

\author[Florio]{Anna Florio$^2$}
\address{$^2$Sorbonne Université, Université de Paris, CNRS, Institut de 
	Mathématiques de Jussieu-Paris Rive Gauche, F-75005 Paris, France}
\email{anna.florio@imj-prg.fr}

\author [Wiseman]{Jim Wiseman$^3$}
\address{$^3$Department of Mathematics, Agnes Scott College, Decatur, Georgia, USA}
\email{jwiseman@agnesscott.edu}

\date{\today}

\begin{document}

\selectlanguage{english}
\maketitle
\begin{abstract}
Let $\phi:X\times\R \rightarrow X$ be a continuous flow on a compact metric space $(X,d)$. In this article we constructively prove the existence of a continuous Lyapunov function for $\phi$ which is strictly decreasing outside $\mathcal{SCR}_d(\phi)$. Such a result generalizes Conley's Fundamental Theorem of Dynamical Systems for the strong chain recurrent set. 

\end{abstract}

\section{Introduction}
\noindent This article continues a project of the authors, started in \cite{BF1} and proceed in \cite{BF2} and \cite{BFW}, concerning the study of the links between recurrent sets and Lyapunov functions. \\
~\newline
Let $\phi$ be a continuous flow  on a compact metric space $(X,d)$. The aim of the present paper is to give a constructive proof of the existence of a continuous Lyapunov function for $\phi$ which is strictly decreasing outside the strong chain recurrent set $\mathcal{SCR}_d(\phi)$. \\
Such a result generalizes Conley's Fundamental Theorem of Dynamical Systems --see the seminal book \cite{conBOOK}[Section 6.4, Page 39]-- since we look at  $\mathcal{SCR}_d(\phi)$, instead of the chain recurrent set $\mathcal{CR}(\phi)$. \\
~\newline
For dynamics given by the iteration of a homeomorphism, the problem has already been  solved by Fathi and Pageault in \cite{FP1}[Theorem 2.6], by using Fathi's formalism in weak KAM theory. In particular, they proved the following result. \\
~\newline
\textit{THEOREM. Let $f:X\rightarrow X$ be a homeomorphism on a compact metric space $(X,d)$. Then there exists a Lipschitz continuous Lyapunov function for $f$ which is strictly decreasing outside $\mathcal{SCR}_d(f)$.} \\
~\newline
Yokoi independently proved the existence of a continuous Lyapunov function (a priori non-Lipschitz) for a homeomorphism on a compact metric space, that strictly decreases outside the $\mathcal{SCR}_d(f)$ (see \cite{Y}[Theorem~5.2]).
~\newline
Combining the variational approach established by Fathi and Pageault in \cite{FP1} and some of Conley's techniques presented in \cite{conBOOK}, the authors Bernardi and Florio have attacked the same problem in the framework of continuous flows and --see \cite{BF2}[Theorem 4.1]-- they proved the next result. \\
~\newline
\textit{THEOREM. Let $\phi:X\times\R\rightarrow X$ be a continuous flow on a compact metric space $(X,d)$, uniformly Lipschitz continuous on the compact subsets of $[0,+\infty)$. Then there exists a Lipschitz continuous Lyapunov function for $\phi$ which is strictly decreasing outside $\mathcal{SCR}_d(\phi)$.} \\
~\newline
The proof of the above result is constructive. Nevertheless, the authors did not manage to generalize the result for only continuous flows, getting rid of the further hypothesis about uniformly Lipschitz regularity of the flow with respect to time. This is due to the fact that, in building the Lyapunov function for the flow, some ``regularizing'' process of an initial function (coming from the variational approach) is needed. In particular, some Lipschitz-like control over time is required. Even if any flow of a Lipschitz continuous vector field satisfies the regularity hypothesis of the above theorem, there are examples of (dynamically interesting) flows that fail to fulfill such a hypothesis, see e.g. Example \ref{es Martin}.\\
~\newline
In this article --for a continuous flow-- we prove the existence of a Lyapunov function, which is strictly decreasing outside the strong chain recurrent set, in the most general framework. In order to obtain it, we exploit the Conley-type decomposition of the strong chain recurrent set, previously shown in \cite{BF1}[Theorem 2]. The constructive method to prove the existence of the required Lyapunov function is then inspired by Conley's original ideas presented in the proof of his celebrated theorem in the chain recurrent case. \\	
~\newline
Thus, our main result reads as follow.
\begin{teorema}\label{main thm}
If $\phi:X\times\R \rightarrow X$ is a continuous flow on a compact metric space $(X,d)$, then there exists a continuous Lyapunov function for $\phi$ which is strictly decreasing outside $\mathcal{SCR}_d(\phi)$.
\end{teorema}
~\newline
\textbf{Acknowledgements.} Olga Bernardi thanks prof. Alberto Abbondandolo who introduced her into Conley's theory of Lyapunov functions and posed her the problem solved in this article. O. Bernardi and J. Wiseman have been supported by the PRIN project 2017S35EHN 003 2019-2021 ``Regular and stochastic behaviour in dynamical systems''.

\section{Preliminaries}
\noindent Let $\phi:X\times\R\rightarrow X$, $(x,t) \mapsto \phi_t(x)$ be a continuous flow on a compact metric space $(X,d)$. In this section we recall the notions of Lyapunov function, strong chain recurrent point, stable set and strongly stable set. 
\smallskip
\begin{definizione}
A continuous function $h:X\rightarrow \R$ is a Lyapunov function for $\phi$ if $h(\phi_t(x)) \leq h(x)$ for every $t\geq 0$ and $x\in X$. 
\end{definizione}
\begin{definizione} Given $x,y\in X$, $\varepsilon>0$ and $T>0$, a strong $(\varepsilon,T)$-chain from $x$ to $y$ is a finite sequence $(x_i,t_i)_{i=1,\dots,n}\subset X\times\R$ such that $t_i\geq T$ for all $i$, $x_1=x$ and, setting $x_{n+1}=y$, we have
$$
\sum_{i=1}^nd(\phi_{t_i}(x_i),x_{i+1})<\varepsilon.
$$
A point $x\in X$ is said to be strong chain recurrent if for any $\varepsilon>0$ and $T>0$ there exists a strong $(\varepsilon,T)$-chain from $x$ to $x$. \\
\noindent The set of strong chain recurrent points is denoted by $\mathcal{SCR}_d(\phi)$.
\end{definizione}
\begin{definizione}
\indent A closed set $B \subset X$ is stable if it has a neighborhood base of forward invariant sets.
\end{definizione}
\noindent We refer to \cite{A}[Page 1732] and \cite{AHK}[Paragraph 1.1]. If $B$ is a stable set, then for every $x \in X$ either $\omega(x) \cap B = \emptyset$ or $\omega(x) \subseteq B$, see \cite{BF1}[Lemma 4.1]. Moreover, the complementary of a stable set $B$ is defined as 
$$B^{\bullet} := \{x \in X: \ \omega(x) \cap B = \emptyset\}.$$ 
The set $B^{\bullet} \subset X$ is invariant and disjoint from $B$ but it is not necessarily closed even if $B$ is closed (see also Paragraph 1.5 in \cite{conARTICOLO}).
\begin{definizione}\label{s s set}
A closed set $B\subset X$ is strongly stable if there exist a family $(U_{\eta})_{\eta\in(0,1)}$ of closed nested neighborhoods of $B$ and a function
	$$
	(0,1)\ni\eta\mapsto T(\eta)\in(0,+\infty)
	$$
	bounded on compact subsets of $(0,1)$, such that:
	\begin{itemize}
		\item[$(i)$] for any $0<\eta<\lambda<1$, $\{x\in X :\ d(x,U_{\eta})<\lambda-\eta\}\subseteq U_{\lambda}$; \vspace{5pt}
		\item[$(ii)$] $B=\bigcap_{\eta\in(0,1)}\omega(U_{\eta})$; \vspace{5pt}
		\item[$(iii)$] for any $0<\eta<1$, $\text{cl}\{ \phi_{[T(\eta),+\infty)}(U_{\eta}) \}\subseteq U_{\eta}$.
	\end{itemize}
\end{definizione}
\noindent We refer to \cite{BF1}[Definition 4.2]. Every strongly stable set $B$ is closed, forward invariant and stable, see \cite{BF1}[Remark 4.1].  \\
\noindent In \cite{BF1}[Theorem 4.2], the subtle relation between strongly stable sets and $\mathcal{SCR}_d(\phi)$ has been explained:

\begin{teorema} \label{teoremann}
	If $\phi:X\times\R\rightarrow X$ is a continuous flow on a compact metric space, then 
	\begin{equation} \label{eccola}
	\mathcal{SCR}_d(\phi)=\bigcap \{ B\cup B^{\bullet} :\ B \text{ is strongly stable} \}.
	\end{equation}
\end{teorema}
\noindent The proof of the main result of this paper is based on the above theorem. 

\section{Proof of Theorem \ref{main thm}}
\subsection{A Lyapunov function for $(B, B^{\bullet})$.} 
~\newline
Let $\phi:X\times\R \rightarrow X$ be a continuous flow on a compact metric space $(X,d)$. \\
In this section, for every pair $(B,B^{\bullet})$ with $B$ strongly stable, we construct a Lyapunov function for $\phi$ which is strictly decreasing on $X\setminus(B\cup B^{\bullet})$. 
\noindent 

 \begin{lemma} \label{Jim}
Let $B \subset X$ be a strongly stable set  and $(U_{\eta})_{\eta\in(0,1)}$ be a family  of closed nested neighborhoods of $B$ as in Definition \ref{s s set}. \\
If $B^\bullet \ne \emptyset$ 
then there exists an $\eta_0 \in (0,1)$ such that
\begin{equation} \label{Jim formula}
B^*_{\eta_0} :=  \{ x\in X :\ \forall t\geq 0,\ \phi_t(x)\notin U_{\eta_0} \}
\end{equation}
is nonempty. The set $B^*_{\eta_0}$ is forward invariant, $B^*_{\eta_0} \subseteq B^{\bullet}$ and $B \cap \text{cl}(B^*_{\eta_0})=\emptyset$.
\end{lemma}

\begin{proof}
Let $x$ be an element of $B^\bullet$, so $\omega(x) \cap B = \emptyset$.  Since  $B=\bigcap_{\eta\in(0,1)}\omega(U_{\eta})$, there exists $\eta_0$ such that $\omega(x) \cap \omega(U_{\eta_0}) = \emptyset$.  Since  $U_{\eta_0}$ is eventually forward invariant, this means that $\phi_t(x) \not\in U_{\eta_0}$ for all $t\ge0$.  
\end{proof}
\noindent 
We observe that, since the $U_\eta$'s are nested,
$$
B^*_{\eta_0}=\{x\in X :\ \forall t\geq 0,\ \phi_t(x)\notin\bigcup_{\eta\in(0,\eta_0]}U_{\eta}\}.
$$
\begin{teorema} \label{OlgaAnna}
	Let $B^\bullet \ne \emptyset$ and $B^*_{\eta_0}$ be as in formula (\ref{Jim formula}) of Lemma \ref{Jim}. \\
	Then there exists a continuous Lyapunov function $h:X\rightarrow\R$ for $\phi$ such that
	
	\begin{itemize}
		\item[$(i)$] $h^{-1}(0)=B$.\vspace{5pt}
		
		\item[$(ii)$] $h^{-1}(1)= \text{cl}(B^*_{\eta_0})$.\vspace{5pt}
		
		\item[$(iii)$] $h$ is strictly decreasing on the set $X\setminus(B\cup B^{\bullet})$.
	\end{itemize}
\end{teorema}

\begin{proof}
	We first define the function $l:X\rightarrow\R$ as follows:
	\begin{equation}
	l(x):=\begin{cases}
	\dfrac{\eta}{\eta_0}\qquad \text{if }\eta = \inf_{\lambda \in (0,\eta_0)} \lambda \text{ such that } x \in U_{\lambda} ,\\
	\\
	1\qquad\text{otherwise.}
	\end{cases}
	\end{equation}
	\noindent The function $l$ is continuous, $B=l^{-1}(0)$, $\text{cl}(B^*_{\eta_0})\subseteq l^{-1}(1)$ and $l(X)\subseteq[0,1]$. Define now the function $k:X\rightarrow\R$ by
	\begin{equation}\label{k def}
	k(x):=\sup_{t\geq 0}\{ l(\phi_t(x)) \}.
	\end{equation}
	\noindent Since both $B$ and $\text{cl}(B^*_{\eta_0})$ are forward invariant, it follows that $B=k^{-1}(0)$ and $\text{cl}(B^*_{\eta_0})\subseteq k^{-1}(1)$. Moreover, $k(X)\subseteq [0,1]$. We show now that the function $k$ is continuous.
	\begin{itemize}
		\item[$(a)$] $k$ is continuous on $\text{cl}(B^*_{\eta_0})$. Since $l$ is continuous and since $\text{cl}(B^*_{\eta_0})\subseteq l^{-1}(1)$, for every $\varepsilon>0$ there exists a neighborhood $V$ of $\text{cl}(B^*_{\eta_0})$ such that
		$$
		\abs{l(y)-l(x)}<\varepsilon
		$$
		$\forall y\in V$ and $\forall x\in \text{cl}(B^*_{\eta_0})$. In particular, for every $y\in V$ we have $1-l(y)<\varepsilon$. Observe that $l\leq k\leq 1$. Thus we deduce that
		$$
		\abs{k(y)-k(x)}=1-k(y)\leq 1-l(y)<\varepsilon
		$$
		$\forall y\in V$ and $\forall x\in \text{cl}(B^*_{\eta_0})$. This concludes the proof of the continuity of $k$ on $\text{cl }(B^*_{\eta_0})$.
		\item[$(b)$] $k$ is continuous on $B$. Since $l$ is continuous and since $B=l^{-1}(0)$, for every $\varepsilon>0$ there exists a neighborhood $V$ of $B$ such that $l_{\vert V}<\varepsilon$. Corresponding to $V$ and up to restricting $V$, there exists $\eta\in(0,\eta_0)$ such that $U_{\eta}\subseteq V$. From property $(iii)$ of Definition \ref{s s set}, there is $T(\eta)>0$ so that
		\begin{equation}\label{cont B}
		\phi_{[T(\eta),+\infty)}(U_{\eta})\subseteq U_{\eta}\subseteq V.
		\end{equation}
		Observe that $\phi_{[T(\eta),+\infty)}(U_{\eta})$ is a neighborhood of $B$. From \eqref{cont B}, the function $k$ is $\varepsilon$-bounded on $\phi_{[T(\eta),+\infty)}(U_{\eta})$. This proves the continuity of $k$ on $B$.
		\item[$(c)$] $k$ is continuous on $X\setminus(B\cup \text{cl }(B^*_{\eta_0}))$. 
	
		Fix $x\in X\setminus(B\cup \text{cl }(B^*_{\eta_0}))$. Thus, there exists $\tau(x)\geq 0$ such that $\phi_{\tau(x)}(x)\in \bigcup_{\eta\in(0,\eta_0)}U_{\eta}$. Let $\bar{\eta}\in(0,\eta_0)$ be such that $\bar\eta=\inf_{\lambda\in(0,\eta_0)}\lambda$ so that $\phi_{\tau(x)}(x)\in U_{\bar{\eta}}$. Then, by property $(i)$ of Definition \ref{s s set} and the continuity of the flow, fixed $ \bar{\eta}<\tilde{\eta} <\eta_0$, there exists a neighborhood  $V$ of $x$ such that $\phi_{\tau(x)}(V) \subseteq U_{\tilde{\eta}}$. Therefore, by property $(iii)$ of Definition \ref{s s set}, we have
		$$\phi_{[\tau(x) + T(\tilde{\eta}),+\infty)}(V)\subseteq U_{\tilde{\eta}}$$
		and consequently
		$$l\vert_{\phi_{[\tau(x) + T(\tilde{\eta}),+\infty)}(V)} \le \dfrac{\tilde{\eta}}{\eta_0}.$$
		
		\noindent Hence, for every $y\in V$ we have
		$$
		k(y)=\sup_{t\geq 0}\{l(\phi_t(y))\}=\max_{t\in [0,\tau(x)+T(\tilde{\eta})]}\{l(\phi_t(y))\}.
		$$
		\noindent By the continuity of $y\mapsto \max_{t\in[0,\tau(x)+T(\tilde{\eta})]}\{l(\phi_t(y))\}$, we conclude that $k$ is continuous at $x\in X\setminus(B\cup \text{cl }(B^*_{\eta_0}))$.
	\end{itemize}
	By its definition in \eqref{k def}, it immediately follows that $k$ is a Lyapunov function for $\phi$. \\
	\noindent Define now the function $h:X\rightarrow\R$ as
	\begin{equation}
	h(x):=\int_0^{+\infty}e^{-s}k(\phi_s(x))\, ds.
	\end{equation}
	\noindent The function $h$ is continuous and decreasing along trajectories. Moreover, on one hand $h(x)=0$ if and only if $k(\phi_s(x))=0$ for every $s\geq 0$, i.e. if and only if $x\in B$. On the other hand, $h(x)=1$ if and only if $k(\phi_s(x))=1$ for every $s\geq 0$. That is if and only if $\phi_s(x)\notin U_{\eta_0}$
	for every $s\geq 0$, which is the definition of $B^*_{\eta_0}$.\\
	\noindent We finally prove that the function $h$ is strictly decreasing on 
	$$x\in X\setminus(B\cup B^{\bullet}) = \{x \in X \setminus B: \ \omega(x) \subseteq B\}.$$ 
	Indeed, for any $x\in X\setminus(B\cup B^{\bullet})$ and $t>0$, we have

	$$h(\phi_t(x)) - h(x) = \int_0^{+\infty}e^{-s}\left( k(\phi_{s+t}(x))-k(\phi_s(x)) \right)\, ds < 0$$
	since the integral is not identically zero. This concludes the proof. 
\end{proof}
\noindent We recall that the corresponding result for an attractor-repeller pair is contained in \cite{conBOOK}[Section B, page 33]. Our proof follows the main lines of Conley's original idea. However we notice that Conley's Lyapunov function for an attractor-repeller pair is identically zero on the attractor and identically 1 on the repeller. The same does not hold for $(B,B^{\bullet})$, with $B$ strongly stable. In such a case --see also Example \ref{funzione 01}  below-- the Lyapunov function of Theorem \ref{OlgaAnna} assumes all the values between $0$ and $1$ in $B^{\bullet}$.
\begin{esempio} \label{funzione 01} (Example 4.4 in \cite{BF1}) \\
\noindent On the circle $\R/\Z$ endowed with the standard quotient metric, consider the dynamical system of Figure \ref{FigBF}. In the sequel, we denote by $\widehat{XY}$ (resp. $cl(\widehat{XY})$) the clockwise-oriented open (resp. closed) arc from $X$ to $Y$. The arc $cl(\widehat{AB})$ and the points $C$ and $D$ are fixed; on the other points we have a clockwise flow. Then $cl(\widehat{AE})$ is strongly stable with $B^{\bullet} = \widehat{ED} \cup \{D\}$. In such a case, a set $B^*_{\eta_0}$ as in formula (\ref{Jim formula}) of Lemma \ref{Jim} is --for example-- $B^*_{\eta_0} = cl(\widehat{BD})$ and the corresponding Lyapunov function for $(B,B^{\bullet})$ constructed in Theorem \ref{OlgaAnna} equals  $0 \text{ if } x \in cl(\widehat{AE})$ and $\frac{d(x,E)}{d(x,E) + d(x,B)} \text{ if } x \in \widehat{EB} \subseteq B^{\bullet}$. Then, such a function assumes all the values between $0$ and $1$ in $B^{\bullet}$ (see Figure \ref{funzione per h}).
\begin{figure}[h]
	\begin{minipage}{0.45\textwidth}
		\centering
		\begin{tikzpicture}
		\draw (4,4) circle (20mm);
		\draw [ultra thick] (4,6) arc (90:135:2);
		\draw plot [mark=*] coordinates
		{(4,6)}; 
		\node at (4,6.3) {B};
		\draw plot [mark=*] coordinates
		{(2.586,5.414)};
		\node at (2.286,5.5) {A};
		\draw plot [mark=*] coordinates
		{(3.1,5.8)};
		\node at (3,6.2){E};
		\draw plot [mark=*] coordinates
		{(2.586,2.586)};
		\node at (2.286,2.5) {D};
		\draw plot [mark=*] coordinates
		{(5.414,2.586)};
		\node at (5.714,2.5) {C};
		\draw [-latex] (4,6) arc (90:45:2);
		\draw [-latex] (6,4) arc (0:-90:2);
		\draw [-latex] (2.586,2.586) arc (225:150:2);
		\end{tikzpicture}
		\caption{The dynamics of Example \ref{funzione 01}.}
		\label{FigBF}
	\end{minipage}
	\begin{minipage}{0.5\textwidth}
		\begin{center}
			\includegraphics[scale=0.32]{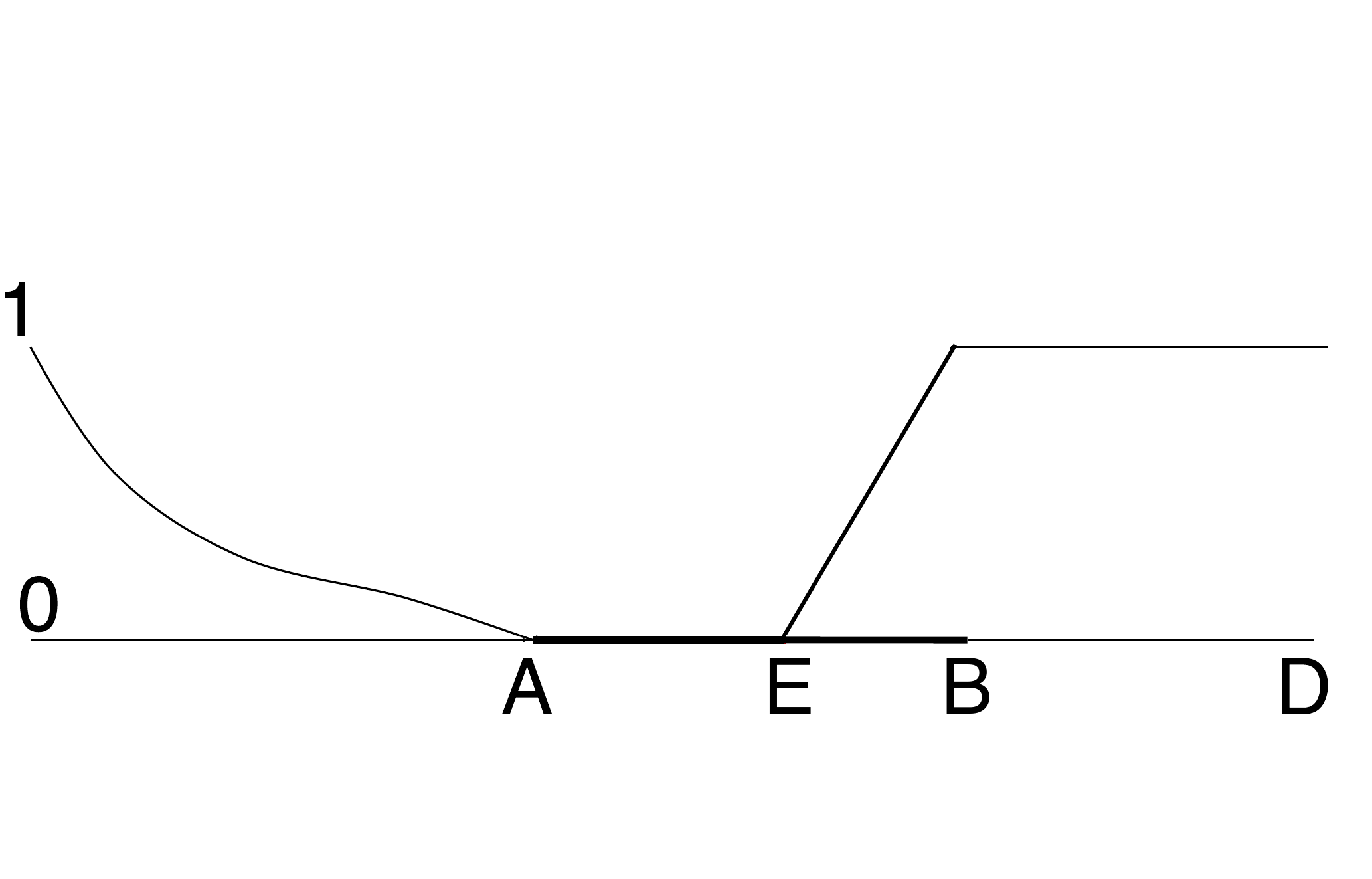}
			\caption{The Lyapunov function for the dynamics of Example \ref{funzione 01}.}
			\label{funzione per h}
		\end{center}
	\end{minipage}
\end{figure} 

\end{esempio}
\subsection{Reduction to a countable set of $(B,B^{\bullet})$.} \label{STABLE}
~\newline
\noindent Conley's proof of Fundamental Theorem of Dynamical Systems uses the facts that  attractor-repeller pairs are closed and that there are countably many such pairs. We notice that the pairs $(B,B^{\bullet})$, involved in the decomposition of $\mathcal{SCR}(\phi)$, aren't as well behaved. Firstly, different $B$'s can give the same $B \cup B^\bullet$: the obvious example is the identity flow on a connected compact metric space. Moreover, $B \cup B^\bullet$ isn't necessarily closed, there can be uncountably many different $B \cup B^\bullet$'s and, finally, the structure $B \cup B^{\bullet}$ is not conserved  by finite intersection. \\
~\newline
\noindent The first part of this section is devoted to showing --by simple examples-- these facts.
\begin{esempio} \label{es union open}
On the compact space $[0,1]^2\subset\R^2$ endowed with the standard metric, consider the dynamical system of Figure \ref{JimFig}.
\begin{figure}[h]
	\centering
	\includegraphics[scale=0.20]{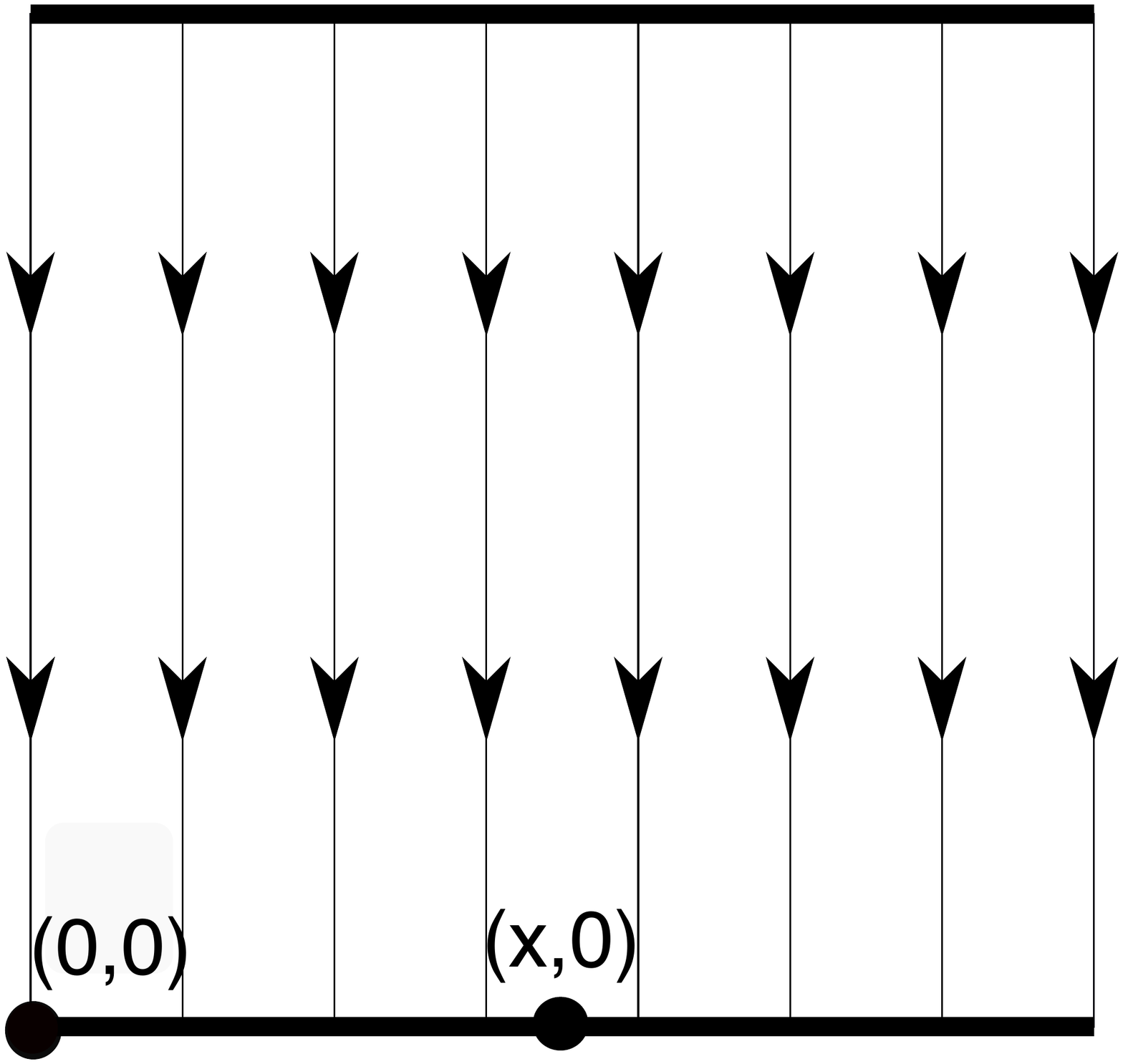}
	\caption{Dynamics of Example \ref{es union open}.}
	\label{JimFig}
\end{figure}
\noindent Segments $[0,1] \times\{0\}$ and $[0,1]\times\{1\}$ are fixed; moreover, for every $x\in[0,1]$, on each open vertical interval $\{x\}\times(0,1)$ we have a north-south flow. In such a case, $B = \{(0,0)\}$ is strongly stable with complementary $B^\bullet = ((0,1]\times[0,1])\cup \{(0,1)\}$, and the union $B\cup B^\bullet$ is not closed. Moreover, every closed segment $B_x$ connecting $(0,0)$ to $(x,0)$, $x \in [0,1]$, is strongly stable, and the collection $(B_x)_{x \in [0,1]}$ gives uncountably many different $B_x \cup B_x^\bullet$'s.
\end{esempio}
\begin{esempio}\label{example intersection} Let us consider again Example \ref{funzione 01}. It is straightforward to see that there exist $B_0$ and $B_1$ strongly stable sets such that
$$B_0 \cup B_0^{\bullet} = cl(\widehat{AD}) \qquad \text{and} \qquad B_1 \cup B_1^{\bullet} = cl(\widehat{DC}).$$
However $cl(\widehat{AD}) \cap cl(\widehat{DC}) = \{D\} \cup cl(\widehat{AC})$ is not the union of a strongly stable set and its complementary.
\end{esempio}

\noindent The aim of the second part of the section is proving the next
\begin{teorema} \label{Teo Jim}
	Let $\phi:X\times\R\rightarrow X$ be a continuous flow on a compact metric space $(X,d)$. Then there exists an --at most countable-- collection of strongly stable sets $\{B_n\}$ such that 
\begin{equation} \label{count decom}
\mathcal{SCR}_d(\phi)=\bigcap_n  B_n\cup B_n^{\bullet}.
\end{equation}
\end{teorema}

\noindent We will use the following lemmas in the proof of the theorem. 

\begin{lemma} \label{primo}
Let $\{U_\alpha\}$ be a --possibly uncountable-- collection of open sets of a separable space.  Then there exists an --at most countable-- subcollection $\{U_n\}$ such that 
$$\bigcup_n U_n = \bigcup_{\alpha} U_\alpha.$$
\end{lemma}

\begin{proof}
Take a countable basis for the topology and let $\{V_n\}$ be the set of basis elements that are contained in $\bigcup_{\alpha} U_\alpha$.
For each $V_n$, set $U_n = U_\alpha$ for some $U_\alpha$ containing $V_n$, or $U_n = \emptyset$ if no such $U_\alpha$ exists.  We claim that $\bigcup_n U_n = \bigcup_{\alpha} U_\alpha$. \\
The inclusion $\bigcup_n U_n \subseteq \bigcup_{\alpha} U_\alpha$ is immediate. To see that $\bigcup_n U_n \supseteq \bigcup_{\alpha} U_\alpha$, let $x$ be any element of $\bigcup_{\alpha} U_\alpha$. Then $x$ is in $U_\alpha$ for some $\alpha$, and so there exists a basis element $V_n$ such that $x \in V_n \subseteq U_\alpha$.  Thus $U_n$ is not empty and contains $x$.
\end{proof}

\begin{lemma} \label{secondo}
Let $\phi:X\times\R\rightarrow X$ be a continuous flow on a compact metric space $(X,d)$. Then 
$$\bigcap \{cl(B\cup B^\bullet) : \text{$B$ is strongly stable}\} = \bigcap \{B\cup B^\bullet: \text{$B$ is strongly stable}\}.$$
\end{lemma}

\begin{proof} The inclusion 
$$\bigcap \{cl(B\cup B^\bullet) : \text{$B$ is strongly stable}\} \supseteq \bigcap \{B\cup B^\bullet: \text{$B$ is strongly stable}\}$$
is immediate. \\
In order to show the other inclusion, we need to recall some results from  \cite{BF1}[Section 3]. \\
For fixed $\varepsilon >0$, $T > 0$ and $Y \subset X$, let
$$\Omega(Y,\varepsilon,T) := \{x \in X : \text{there is a strong } (\varepsilon,T)\text{-chain from a point of } Y \text{ to } x\}$$ and
$$\bar{\Omega}(Y,\varepsilon,T) := \bigcap_{\eta > 0} \Omega(Y,\varepsilon + \eta,T). $$ 
Moreover, let
$$\bar{\Omega}(Y) := \bigcap_{\varepsilon >0, \ T > 0} \Omega(Y,\varepsilon,T) = \bigcap_{\varepsilon >0, \ T > 0} \bar{\Omega}(Y,\varepsilon,T).$$
Let now $x$ be a point in $cl(B\cup B^\bullet)\backslash(B\cup B^\bullet)$ for some strongly stable set $B$.  We will show that there exists a strongly stable set $\tilde{B}$ (clearly depending on the point $x$) such that $x\not\in cl(\tilde{B}\cup \tilde{B}^\bullet)$. \\
Since $x \notin B\cup B^\bullet$, by Theorem \ref{teoremann}, $x \notin \mathcal{SCR}_d(\phi)$. Equivalently, 
$$x \notin \bar{\Omega}(x) = \bigcap_{\varepsilon > 0, \ T > 0} \bar{\Omega}(x,\varepsilon,T).$$
In particular, there exist $\varepsilon > 0$ and $T >0$ such that $x \notin \bar{\Omega}(x,\varepsilon,T)$. \\
We proceed by proving that --corresponding to these $\varepsilon > 0$ and $T >0$-- there exists a closed ball $\tilde{C}$ centered in $x$ such that
\begin{equation} \label{ball}
\tilde{C} \cap \bar{\Omega}(\tilde{C}, \varepsilon, T) = \emptyset.
\end{equation}
Suppose, for the sake of contradiction, that $C \cap \bar{\Omega}(C, \varepsilon, T) \ne \emptyset$ for every closed ball $C$ centered at $x$, in particular, for a sequence $(C_n)_{n \in \mathbb{N}}$ of closed balls centered at $x$ of radius $1/n \to 0$. This means that for every $\eta > 0$ and $n \in \mathbb{N}$ there are points $x_n^{\eta}$ and $y_n^{\eta} \in C_n$ such that there exists a $(\varepsilon+\eta,T)$-chain from $y_n^{\eta}$ to $x_n^{\eta}$. Since, for every $\eta > 0$, 
$$\lim_{n \to + \infty} x_n^{\eta} = x = \lim_{n \to +\infty} y_n^{\eta},$$
we conclude that $x \in \bar{\Omega}(x,\varepsilon,T)$. This fact contradicts the hypothesis that $x \notin \bar{\Omega}(x,\varepsilon,T)$ and therefore there necessarily exists a closed ball $\tilde{C}$ centered in $x$ satisfying formula (\ref{ball}). \\ 
In order to conclude,  define
$$\tilde{B} := \omega(\bar{\Omega}(\tilde{C},\varepsilon,T)),$$
which is --see Example 4.3 in \cite{BF1}-- a strongly stable set. Moreover (see Corollary 3.1 in \cite{BF1}),
$$\tilde{B} \subseteq \bar{\Omega}(\tilde{C},\varepsilon,T).$$
As a consequence, since (by formula (\ref{ball})) $x \notin \bar{\Omega}(\tilde{C},\varepsilon,T)$, then $x \notin \tilde{B}$. We finally recall that --by Lemma 3.6 in \cite{BF1}-- $\omega(\tilde{C}) \subseteq  \omega(\bar{\Omega}(\tilde{C},\varepsilon,T))$, so that for every point $y \in \tilde{C}$, we have 
$$\omega(y) \subseteq \omega(\tilde{C}) \subseteq \omega(\bar{\Omega}(\tilde{C},\varepsilon,T)) = \tilde{B}.$$
This means that $y \notin \tilde{B}^{\bullet}$ for all $y \in \tilde{C}$. \\
Finally, since the point $x \notin \tilde{B}$ and $x \in int(\tilde{C})$, we conclude that $x\not\in cl(\tilde{B}\cup \tilde{B}^\bullet)$ and the desired inclusion is proved.  
\end{proof}

\noindent As a direct consequence of the previous lemmas, we can give the \\
~\newline
\noindent \textit{Proof of Theorem \ref{Teo Jim}.}  Consider the collection of open sets
$$\{U_\alpha\} := \{ X\backslash cl(B_\alpha\cup B^\bullet_\alpha) : \text{$B_\alpha$ is strongly stable} \}.$$ 
 Then, by Lemma \ref{primo}, there exists an --at most countable-- subcollection $\{U_{n}\}$ such that
 $$\bigcup_{\alpha} U_\alpha = \bigcup_{n} U_n.$$
Consequently, applying also Theorem \ref{teoremann} and Lemma \ref{secondo}, we obtain
\begin{eqnarray*}
\mathcal{SCR}_d(\phi) &=& \bigcap_{\alpha} B_{\alpha} \cup B_{\alpha}^{\bullet} = \bigcap_{\alpha} cl(B_{\alpha} \cup B_{\alpha}^{\bullet}) \\
&=& X \backslash \bigcup_{\alpha} U_\alpha = X \backslash \bigcup_n U_n = \bigcap_n cl(B_n\cup B_n^{\bullet}) \\
&\supseteq&  \bigcap_n B_n\cup B_n^{\bullet} \supseteq \bigcap_\alpha B_\alpha\cup B_\alpha^{\bullet} = \mathcal{SCR}_d(\phi),
\end{eqnarray*}
which gives exactly formula (\ref{count decom}). \hfill $\Box$

\subsection{End of proof of Theorem \ref{main thm}.} 
~\newline
\noindent\textit{Proof of Theorem \ref{main thm}.} With reference to Theorems \ref{OlgaAnna} and \ref{Teo Jim}, let denote by $h_n$ the Lyapunov function for the $n$th-pair $(B_n,B_n^{\bullet})$. As in \cite{conBOOK}[Page 39], define the function
	\begin{equation}
	h(x)=\sum_{n=0}^{+\infty}\dfrac{h_n(x)}{3^n}.
	\end{equation}
	The function $h$ is continuous. Moreover --as a consequence of Theorems \ref{teoremann} and \ref{OlgaAnna}-- $h$ is a Lyapunov function for $\phi$ which is strictly decreasing outside $\mathcal{SCR}_d(\phi)$.
	\hfill\qed \\
	\begin{remark}
	We observe that the proof in \cite{BF1} of the decomposition of the strong chain recurrent set (here Theorem \ref{teoremann}) actually  uses only the property of $\phi$ being a semiflow. Moreover,  the construction of the function in Theorem~\ref{OlgaAnna} and the countability result in Theorem~\ref{Teo Jim} also work under this hypothesis. Consequently, Theorem~\ref{main thm} can be rephrased, more generally, for a continuous semiflow on a compact metric space.
\end{remark}	
	~\newline
\noindent As recalled in the Introduction, if the flow is uniformly Lipschitz continuous on compact subsets of $[0,+\infty)$, then the result of Theorem \ref{main thm} is proved in \cite{BF2}[Theorem 4.1]. The proof we present here does not need any additional regularity assumption on the continuous flow. \\
~\newline
We finally present an example of a continuous flow which is not uniformly Lipschitz continuous on compact time subsets. That is, an example of a flow to which the result in \cite{BF2} cannot be applied, but for which Theorem \ref{main thm} holds.

	\begin{esempio}\label{es Martin} 	
	\noindent Let us think of $\mathbb{T}$ as $[0,1]$ with $0$ identified to $1$. Define $\tau:[0,1]\rightarrow\R$ as
	$$
	\tau(x):=\begin{cases}
	\sqrt{x-\frac{1}{2}}+\frac{\sqrt{2}-1}{\sqrt{2}}\quad x\in\left[\frac{1}{2},1\right] \\
	\\
	\sqrt{\frac{1}{2}-x}+\frac{\sqrt{2}-1}{\sqrt{2}}\quad x\in\left[0,\frac{1}{2}\right]
	\end{cases}
	$$
Consider
	$
	X=\{ (x,y)\in\R^2 :\ x\in[0,1], \ 0\leq y\leq \tau(x) \}
	$
	endowed with the standard metric from $\R^2$, see Figure \ref{funz-tetto}. Identify the graph of $\tau$ with $[0,1]$, i.e. such that $(x,\tau(x))\sim(x,0)$.
 Let us consider the flow $\phi$ on $X$ whose flow lines are described in Figure \ref{flow}. 
 In particular, the central vertical strip (of positive measure) is made of periodic points. Consequently, $\mathcal{SCR}(\phi) = Per(\phi) \ne X$. This is an example of a continuous flow, not uniformly Lipschitz continuous on compact subsets of $[0,+\infty)$. 
  \begin{figure}
	\begin{minipage}{0.45\textwidth}
		\begin{center}
			\includegraphics[scale=0.3]{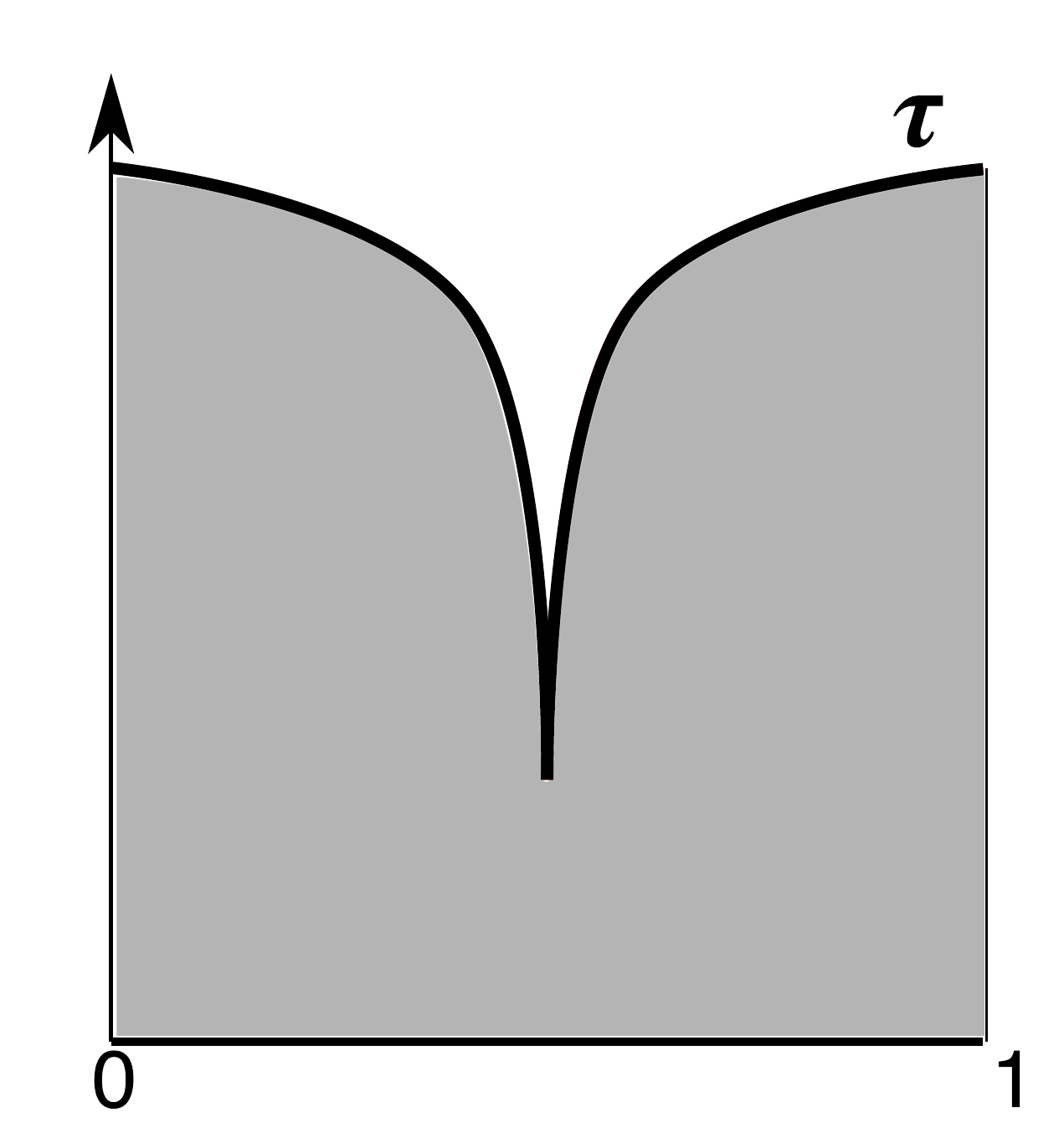}
			\caption{The metric space of Example \ref{es Martin}.}
			\label{funz-tetto}
		\end{center}
	\end{minipage}
	\begin{minipage}{0.45\textwidth}
		\begin{center}
			\includegraphics[scale=0.16]{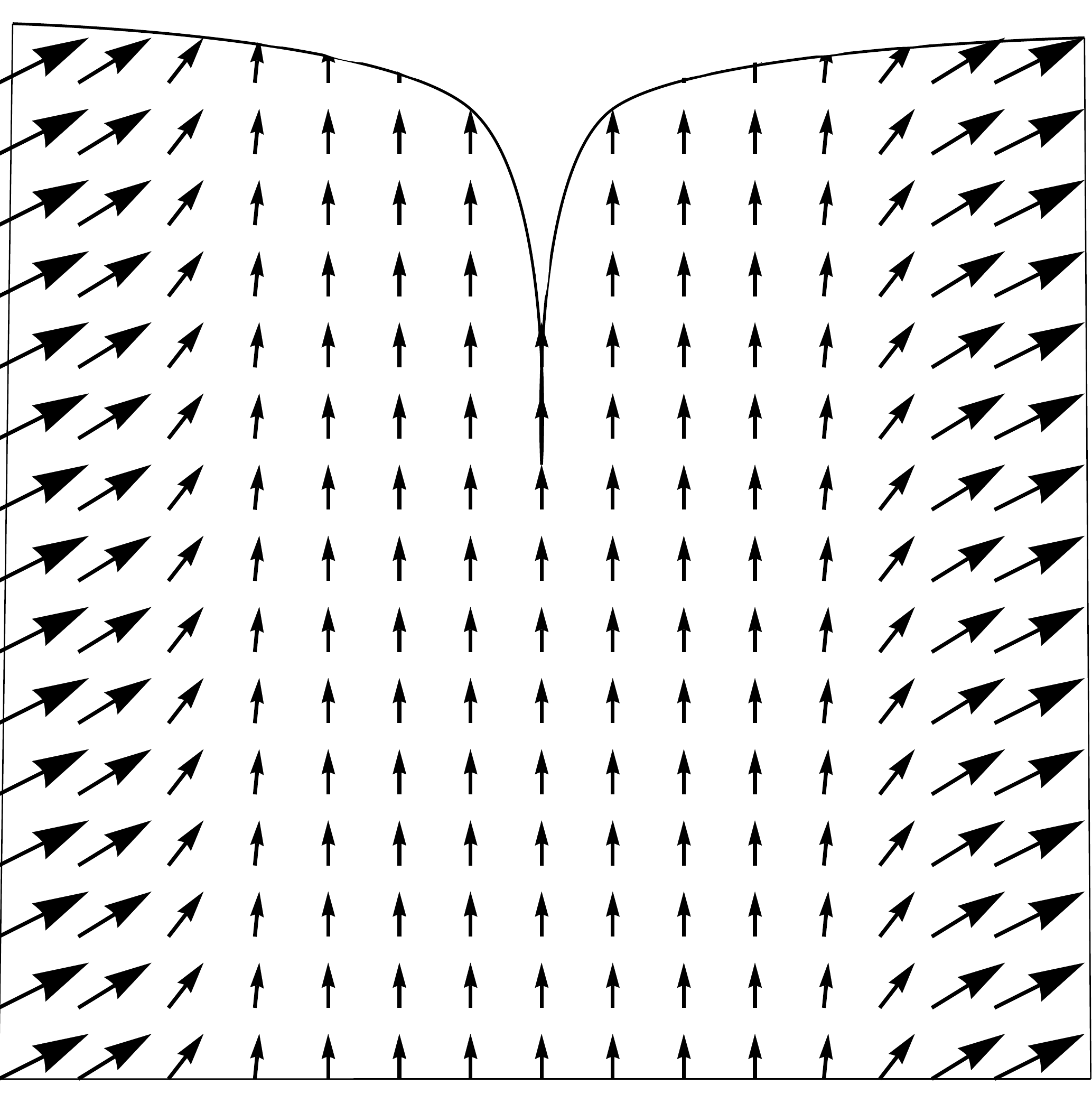}
			\caption{The flow lines of Example \ref{es Martin}.}
			\label{flow}
		\end{center}
	\end{minipage}
\end{figure}
\end{esempio}

\bibliographystyle{alpha}

\end{document}